\documentclass{imsart}

\usepackage{amsthm,amsmath,amssymb}
\usepackage[usenames,dvipsnames]{xcolor}
\usepackage{subfigure}
\usepackage{graphicx}
\RequirePackage{natbib}
\RequirePackage[colorlinks,citecolor=blue,urlcolor=blue]{hyperref}
\usepackage{longtable}

\startlocaldefs

\newcommand{\Z}{\mathbb Z}
\newcommand{\N}{\mathbb N}
\newcommand{\F}{\mathcal F}

\newcommand{\R}{\mathbb R}
\newcommand{\E}{\mathbb E}
\newcommand{\p}{\mathcal P}
\newcommand{\T}{\mathcal T}

\newtheorem{theorem}{Theorem}[section]
\newtheorem{lemma}[theorem]{Lemma}

\newtheorem{proposition}[theorem]{Proposition}
\newtheorem{definition}[theorem]{Definition}
\theoremstyle{remark}

\newtheorem{example}[theorem]{Example}

\numberwithin{equation}{section}
\endlocaldefs

\begin{document}

\begin{frontmatter}

\title{Nonparametric statistical inference for the context tree of a stationary 
ergodic process\thanksref{t1}}

\thankstext{t1}{This article was produced as part of the activities of FAPESP  Research, Innovation and Dissemination Center for Neuromathematics, grant
2013/07699-0, S\~ao Paulo Research Foundation. It has also received financial support 
from the projects \emph{Stochastic systems: equilibrium and non-equilibrium, limits in scale and percolation}, grant CNPq 474233/2012-0, and \emph{Stochastic chains of long range}, grant FAPESP 2015/09094-3.}

\runtitle{Nonparametric statistical inference for context trees}

\author{\fnms{Sandro} \snm{Gallo}
\ead[label=e1]{sandro.gallo@ufscar.br}}
\address{Mathematics Departement\\Federal University of S\~ao Carlos, Brazil\\ \printead{e1}}

\author{\fnms{Florencia} \snm{Leonardi}\thanksref{t2}\corref{}\ead[label=e3]{florencia@usp.br}
\ead[label=e4,url]{http://www.ime.usp.br/$\sim$leonardi}}
\address{Institute of Mathematics and Statistics\\University of S\~ao Paulo, Brazil\\ \printead{e3}}

\thankstext{t2}{Partially supported by a CNPq-Brazil fellowship 304836/2012-5 and 
a L'Or\'eal Fellowship for Women in Science. } 

\runauthor{S. Gallo and F. Leonardi}

\begin{abstract}
We consider the problem of estimating the context tree of a  stationary ergodic process with finite alphabet without imposing additional conditions on the process.
As a starting point we introduce a  \emph{Hamming} metric in the space of irreducible context trees and we use the properties of the weak topology in the space of ergodic stationary processes to prove that if the Hamming metric is unbounded, there exist no consistent estimators for the context tree. Even  in the bounded case we show that there exist no two-sided confidence bounds. However we prove that one-sided inference is possible in this general setting and we construct a consistent estimator that is a lower bound for the context tree of the process with an explicit formula for the  coverage probability.
We develop an efficient algorithm to compute the lower bound and we apply the method to test a linguistic hypothesis about the context tree of codified written texts in European Portuguese. 
\end{abstract}

\begin{keyword}[class=MSC]
\kwd[Primary ]{62M09}
\kwd{62G15}
\kwd{62G20}
\kwd[; secondary ]{60G10}
\kwd{60J10}
\end{keyword}

\begin{keyword}
\kwd{variable length markov chain}
\kwd{context tree}
\kwd{confidence bounds}
\kwd{consistent estimation}
\kwd{nonparametric inference}
\end{keyword}


\end{frontmatter}

\section{Introduction}

In this work we address the issue of whether or not there exist consistent estimators (and confidence bounds) for the \emph{context tree} of a discrete time stationary ergodic process with finite alphabet. 
In words, the context tree of a stochastic process is a set of finite strings or left-infinite sequences that determines the portion of the past the process has to look at in order to decide the distribution of its next symbol.  For example, an i.i.d. process has the empty string as context tree since it has no dependence on the past. A $k$-steps Markov chain has a context tree containing at least one string of length $k$,  and a non-Markovian chain (sometimes coined infinite memory process) has a context tree having at least one left-infinite sequence. 

Finite context trees were introduced by \citet{rissanen1983} as an efficient tool for data compression. The corresponding processes were originally called \emph{Variable length Markov Chains} (VLMC) and its estimation was first addressed in \citet{buhlmann1999}. Recently, they have received increasing attention in the applied statistics literature, being used in a wide range of problems from different areas \citep[for instance]{bejerano2001a,dalevi-et-al2006,busch-et-al2009,galves-et-al2012}. Its success  in real word applications seems to stem from its parsimony (including memory only where data needs) and its capacity to capture structural dependencies in the data. The counterpart of the model, when compared to finite step Markov models for instance,  is that  estimation is a much complicated task. 
When he introduced the model, \cite{rissanen1983} also provided an algorithm for recovering the  context tree out of a given sample. Since then, a large part of the related statistical literature has focussed on consistent estimation of the context tree in the finite and infinite memory case, an incomplete list includes \cite{buhlmann1999,galves-leo-2008,collet-et-al2008,csiszar2006,garivier2011}. 

Most of the above cited works make some assumptions on the processes, such as lower bounding the transition probabilities or imposing mixing conditions, additionally to ergodicity. In the present paper, we precisely refer to our statistical inference problem as \emph{nonparametric} because we make no further assumptions concerning the distribution of the process, else than ergodicity.
 In this nonparametric setting,  \citet{csiszar2006} proved the consistency of the \emph{Bayesian Information Criterion} (BIC)  when the context trees are truncated to a given finite length (the truncation being necessary only for infinite context trees). Interestingly, nothing has been done  concerning confidence bounds as far as we know.

Given a sample of a stationary ergodic process, it is natural to wonder whether this process has a  finite or infinite context tree.  This cannot be consistently decided in this general class  \citep{bailey-phd,morvai-class}.  That is, there exists no two-valued function of the sample which, as the sample increases, stabilizes to the value ``yes'' for every process having a finite context tree and ``no'' for every process having an infinite context tree. Thus, when considering the discrete metric in the space of trees, the existence of a universal consistent estimator relies on assumptions that cannot be checked  empirically. 
This situation has its counterpart in nonparametric statistics for i.i.d observations. For instance, \cite{fraiman1999} observed that it is impossible to decide, out of a random sample, whether or not the underlying distribution has a finite number of modes. Assuming a priori that the number of modes is finite, they can be consistently estimated. \\

In the present work the space of irreducible context trees with finite alphabet is equipped with the Hamming distance. Using only topological arguments 
 we prove that if this metric in the space of trees is unbounded, there exists no  consistent estimator of the context tree in the class of stationary  ergodic processes. In the bounded metric case, we construct an estimator that is consistent and also a nonparametric lower bound with an explicit coverage probability, based on a result of \citet{garivier2011}. 
Finally, following \citet{donoho1988}, we also prove  that it is not possible to obtain nonparametric upper bounds even in the smaller class of processes having finite context trees. To our knowledge, this is the first work considering  the problem of construction of nonparametric confidence bounds for context trees.  

Notation, definitions and main results are given in the next section. In Section~\ref{computation} we show how to compute the lower confidence bound and we present a practical application, testing a linguistic hypothesis about the memory of stressed and non-stressed syllables in European Portuguese written texts. The proofs of the results are given in Section \ref{sec:proofs}. 
  
\section{Definitions and results}\label{sec:notation}

In this section  we present the main definitions and theoretical results  of this paper. 
We begin by describing the notion of irreducible tree and we introduce a \emph{Hamming} distance in the set of all irreducible trees over a finite alphabet.  Then we proceed by defining the context tree of a stationary  ergodic process and by establishing some topological  properties of the set of all stationary ergodic probability measures with respect to the weak topology. The last part of the section is dedicated to the statements of the main results of the paper.

\subsection{Metric tree space}

 Let $A$ be a finite set called alphabet.  For any $m\leq n$, we denote by $a_{m}^{n}$ the string $a_{m}\ldots a_n$ of symbols in $A$ with length $n-m+1$. This notation is also valid for $m=-\infty$ in which case we obtain a left-infinite sequence $a_{-\infty}^{n}$. If $m>n$ we let $a_m^n$ denote the empty string $\lambda$. 
 The length of a string $w$ will be denoted by $|w|$. For any $j\in \{0,1,\ldots\}$, we let $A^j$ denote the set of strings in $A$ having length $j$, in particular $A^0=\{\lambda\}$.  We also let $A^\star=\cup_{j\ge0}A^j$ denote the set of all finite strings on $A$ and we denote by $A^{\infty}$ the set of all left-infinite sequences $a_{-\infty}^{n}$ with symbols in $A$.
 
  We will need to concatenate strings; for instance, if $v\in A^i$ and $w\in A^j$ are strings of length $i$ and $j$ respectively, then $vw$ denotes the string of length $i+j$ obtained by putting the symbols in $w$ after the ones in $v$.
  We also extend concatenation to the case where $v\in A^\infty$ is an infinite string on the left. We say that $w$ is a \emph{suffix} of the sequence $s$ if there exists a sequence $v$ such that $s=vw$. When $|v|\geq 1$ we say that $w$ is a proper suffix of $s$.

A \emph{tree} $\tau$ is any set of strings or perhaps of left-infinite sequences, called \emph{leaves}, such that no $w\in\tau$ is a proper suffix of any other $s\in\tau$. This property enables us to represent the set $\tau$ as a graphical rooted tree by identifying the elements in $\tau$ with \emph{paths} from the terminal nodes of the tree to the root. As an example of finite tree, consider the set $\tau_1 = \{00,010,110,1\}$ over the alphabet $A=\{0,1\}$.  On the other hand, an example of an infinite tree over $A$ is given by 
 $\tau_2=\{10_1^i\colon i=0,1,\dotsc\}\cup \{0^\infty\}$, which has a unique infinite element, the left-infinite sequence $0^\infty$. The graphical representation of these trees can be found in Fig.~\ref{exampletrees}.
Special cases of trees are given by  the entire set
$A^\infty$ of left-infinite sequences, denoted in this paper by $\tau^\infty$, and the tree consisting of the unique empty string $\lambda$, denoted by $\tau^\text{root}$.

\begin{figure}
\center\includegraphics[scale=0.7]{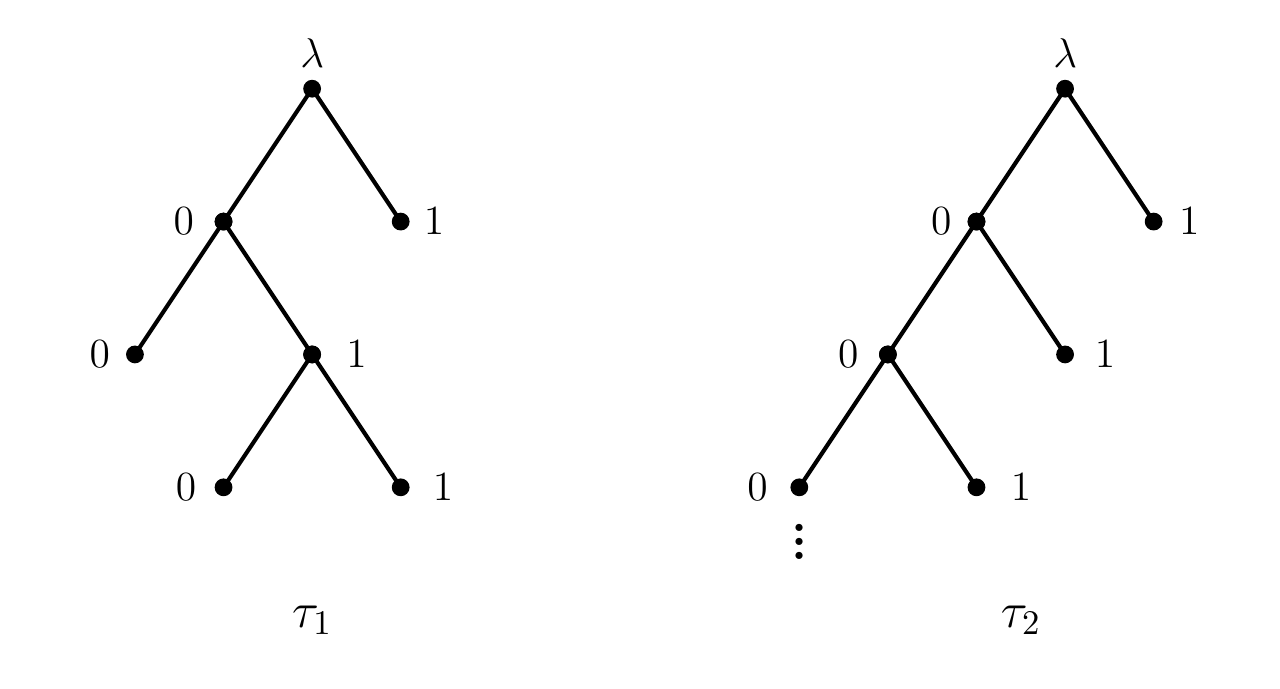}
\caption{Graphical representation of the trees $\tau_1 = \{00,010,110,1\}$ and 
$\tau_2=\{10_1^i\colon i=0,1,\dotsc\}\cup\{0^\infty\}$. In both cases, the contexts corresponds to 
the sequences obtained by concatenating the symbols from the leaves to the root of the trees. In the case of $\tau_2$ we only show  the strings of length at most 3.}
\label{exampletrees}
\end{figure}

We say that the tree $\tau$ is \emph{irreducible} if no $w\in\tau$ can be replaced by a proper suffix without violating the tree property. Both trees in Fig.~\ref{exampletrees} are irreducible, as well as  $\tau^\infty$ and $\tau^\text{root}$. An example of a non-irreducible tree is $\tau_3 = \{000,010,110,1\}$, because substituting 000 by 00 leads to $\tau_1$ that satisfies the tree property.

We will call  a \emph{node} of $\tau$ any finite string that is a  suffix of some $s\in\tau$. Sometimes it will be convenient to identify $\tau$ with the set of its nodes $\bar\tau\subset A^\star$.  In fact it is easy to verify that $\tau$ uniquely determines $\bar{\tau}$ and \emph{vice versa}.  In the case of $\tau_1$ given before, the set  $\bar\tau_1$ is the set of all strings represented in Fig.~\ref{exampletrees}, that is $\bar\tau_1=\{010,110,00,10,0,1,\emptyset\}$.  
In the case of $\tau_2$ we have $\bar\tau_2=\tau_2\cup\{0_1^i\colon i=0,1,\dotsc\}$.

Let $\T$ denote the set of all irreducible trees on $A$, with the following partial order
\[
\tau\;\prec (\preceq) \;\tau'\quad \text{  if and only if  }\quad \bar\tau\;\subsetneq(\subseteq)\; \bar\tau'.
\]
Given a tree $\tau\in \T$ and a constant $k\in\N$, we denote by 
 $\tau|_k$
the truncated tree at level $k$, defined by the set of its nodes 
\[
\bar\tau|_k =\{ v\in\bar{\tau}\colon |v| \leq k\}\,.
\]
Finally, $\T$ is equipped with   the \emph{Hamming} distance defined by 
\begin{equation}\label{hamming}
d_\phi(\tau,\tau') \;=\; \sum_{v\in A^\star} \phi(v)\;| \mathbf{1}_{\{v\in \bar{\tau}\}} - \mathbf{1}_{\{v\in \bar{\tau}'\}} |\;,
\end{equation}
where $\phi\colon A^\star\to \R^+$. In the summable case $\sum_{v\in A^\star}\phi(v)<+\infty$ we have that $(\T,d_\phi)$ is a bounded metric space.

\subsection{Context tree of a stationary ergodic process}

Let $\{X_i\colon i\in\Z\}$ be a stationary and ergodic process assuming values in the alphabet $A$. We denote by $P(a_n^m)$ the stationary probability of the string $a_n^m$, that is 
\[
P(a_n^m)\;=\; \text{Prob}\bigl(X_n^m= a_n^m\bigr)\,.
\]
If $s\in A^\star$ is such that $P(s)>0$ we write
\[
P(a|s) \;=\; \text{Prob}\bigl(X_0=a\,|\, X_{-|s|}^{-1}=s\bigr),
\]
with the convention that if $s=\emptyset$ then $P(a|s)= \text{Prob}\bigl(X_0=a\bigr)$.

A process as above is said to have law, or measure, $P$.

\begin{definition}\label{defcontext}
We say that the string $s\in A^\star$ is a \emph{context} for a process with measure $P$ if it satisfies
\begin{enumerate}
\item $P(s)>0$ or $s=\emptyset$\,.
\item For all $a\in A$ and all $w\in A^\star$ such that $s$ is suffix of $v$ 
\begin{equation}\label{eq:abuse}
\text{Prob}\bigl(X_0=a\,|\, X_{-|v|}^{-1}=v\bigr)\;=\; P(a|s)\,.
\end{equation}
\item No proper suffix of $s$ satisfies 2. 
\end{enumerate}
An \emph{infinite context} is a left-infinite sequence $x_{-\infty}^{-1}$ such that its finite suffixes $x_{-n}^{-1}, n=1,2,\dotsc$ have positive probability but none of them is a context.  
\end{definition} 

By this definition, the set of contexts of a process with measure $P$ is an irreducible tree, it will be denoted by $\tau_P$. 

\begin{example}
Consider the stationary   Markov chain of order 3 over the alphabet $A=\{0,1\}$ defined  by the transition probabilities
\begin{center}
\begin{tabular}{|r|c|c|}
\hline
$w$ & $P(0|w)$ & $P(1|w)$ \\\hline
$ab1$ & 0.2 & 0.8 \\\hline
$a00$ & 0.5 & 0.5  \\\hline
$010$ & 0.3 & 0.7 \\\hline
$110$ & 0.7 & 0.3 \\\hline 
\end{tabular}
\end{center}
where $a,b \in A$ are arbitrary. This is an example of what is called a \emph{Variable Length Markov Chain} (VLMC). 
By Definition~\ref{defcontext}, the only contexts of this process are the strings 1, 00, 010 and 110. The context tree
$\tau_P$ is the tree $\tau_1$ represented in Fig.~\ref{exampletrees}.

\end{example}

\begin{example}
Suppose that the process $\{X_i\colon i\in\Z\}$ takes values in $\{0,1\}$, and in order to decide the probability distribution of the next symbol based on the past realization, we only need to know the distance to the last occurrence of a $1$.
 Then, for any $k\ge0$, any $i\ge1$ and any $v,w\in A^i$
\[
P(1|v10^k)=P(1|w10^k).
\]
According to Definition \ref{defcontext}, the strings $10^k$, $k\ge0$, as well as the semi-infinite sequence $0^\infty$ are context of this process. Therefore, the context tree $\tau_P$ is $\tau_2$ shown in Fig.~\ref{exampletrees}. 
\end{example}

\subsection{The weak topology in the space of stationary ergodic processes}

Let $\Sigma$ be the  $\sigma$-algebra on $\Omega=A^\Z$ obtained as the product of the discrete $\sigma$-algebra on $A$. 
Let $\p$ denote the set of all stationary ergodic probability measures over $(\Omega,\Sigma)$.

Define the following distance in $\p$
\[
D(P,Q) = \sum_{k\in\N} 2^{-k} |P - Q |_k\,,
\]
where 
\[
|P - Q |_k = \sum_{a_1^k\in A^k} |P(a_1^k) - Q(a_1^k)|
\]
is the $k$-th order variational distance. This distance is known in the literature as the \emph{weak distance}, and the topology induced by it is known as the \emph{weak topology} \cite[Section I.9]{shields1996}.  

We now state a basic lemma about the topological properties of  the space $\p$ with respect to the weak topology.

\begin{lemma}\label{lema_baire}
The space $(\p,D)$ is a Baire space. 
\end{lemma}

\subsection{Consistent estimation and confidence bounds}

As mentioned in the Introduction, in this paper we are interested in 
the estimation of properties of the context tree $\tau_P$ from samples $X_1,\ldots,X_n$ of size $n$ of the corresponding stationary and ergodic process $P$.
Up to now this problem has  been reduced to the consistent  identification of the set of contexts (in the finite case) or of a truncated version of the context tree (in the infinite case). The latter corresponds to a special case of our distance $d_\phi$; for instance when the interest is in estimating contexts of length at  most $k$ we can consider $\phi(v)=0$ for all $|v|>k$.  In the sequel we define  the notion of consistency of a sequence of estimators in a general setting.\\

  Let $F:\p\rightarrow \F$ be a functional with values in some metric space $(\F,d)$.
 
 \begin{definition}
 We  say that  $F$ is consistently estimable on $\p$ (in probability) if there exists a sequence 
 $\{F_n\}_{n\in\N}$ of statistics, with $F_n\colon A^n\to\F$, 
such that for all $P\in\p$
\[
d\left(F_n(X_1,\dotsc,X_n),F(P)\right)\;\stackrel{P}{\longrightarrow}\;0\,.
\]
\end{definition}
In this case we say that $\{F_n\}_{n\in\N}$ is a consistent estimator for $F$ on $\p$. We say 
that $F$ is \emph{strongly} consistent on $\p$ if the convergence takes place almost surely 
with respect to the probability measure $P$, and in this case we say that $\{F_n\}_{n\in\N}$ is 
a strongly consistent estimator for $F$ on $\p$.

The following result establishes a necessary condition for the existence of consistent 
estimators of a  bounded real  functional defined on $\p$.

\begin{proposition}\label{non_consistency}
Assume $F\colon \p\to \R$ is bounded (that is there exists $R\in\R$ such that 
$|F(P)| \leq R$ for all $P\in\p$).  
If  $F$ is consistently estimable on $\p$ then $F$ must be continuous  on  a dense subset of $\p$.
\end{proposition}

In this paper we are  concerned with the functional $T\colon\p\to\T$ that assigns to any 
measure $P\in\p$ its associated context tree $\tau_P\in\T$.  
The first question we address here  is if it is possible to decide, out from a finite sample, if the 
sum  of the function $\phi$ over the nodes of the context tree is finite or not. 

\begin{theorem}\label{discontinuity}
If $\sum_{v\in A^\star} \phi(v) = +\infty$ then the functional 
\[
L(P)={\bf 1}\{\text{\small $\sum_{v\in \bar\tau_P} \phi(v)< +\infty$}\}  
\]
is not consistently estimable on $\p$. 
\end{theorem}

This result states, in particular, that the functional that attributes the value $1$ if the measure is Markovian, and $0$ otherwise, is not consistently estimable when $\phi$ is not summable. This is a known result; see \citet{morvai-class} and references therein. However, our proof is completely different from theirs and it is mainly based on topological properties of $\p$.  

Our main result about consistent estimation for the context tree on $\p$ is given in the 
following theorem. 

\begin{theorem}\label{teo-main}
$T$ is  consistently estimable on $\p$ if and only if $\sum_{v\in A^\star} \phi(v)$ is 
finite.  
\end{theorem}

The \emph{only if} part of this theorem is a direct consequence of Theorem \ref{discontinuity}. 
The  \emph{if} part is proved constructively later, because the estimator $\{T_n^c\}_{n\in\N}$ 
defined  by \eqref{eq:estimator} below will be proved to be consistent when $\phi$ is 
summable. 

As mentioned before, the present work is also concerned with the obtention of  confidence bounds for the context tree of a stationary and ergodic  
process. We use the following  general definition of upper and lower confidence bounds, taken from 
\cite{donoho1988}. Suppose $\F$ is equipped with a partial order $<$  with supremum and infimum. 

\begin{definition}
Given $n\ge1$, a statistic $U_n\colon A^n\to\F$ is called a  non-trivial upper confidence bound for $F$ on $\p$ with coverage probability at least $1-\alpha$ if
\[
\sup_{P\in\p}P(U_n  < \sup_{P'\in\p} F(P'))= 1
\]
and
\[
\inf_{P\in\p} P(F(P) \leq U_n)\;\geq\;1-\alpha.
\]
Analogously we say that $L_n$ is a non-trivial lower confidence bound for $F$ on
$\p$   with coverage probability at least $1-\alpha$ if $-L_n$ is a non-trivial upper 
confidence bound for $-F$ on $\p$  with coverage probability at least $1-\alpha$.
\end{definition}

Our first theorem concerning confidence bounds is a negative result stating that the 
functional $T$ does not admit a non-trivial upper confidence bound  neither on $\p$ nor in the class of stationary ergodic measures with finite context tree.\\

\begin{theorem}\label{teo_upper_bound}
If $U_n$ is an upper bound that satisfies 
$\sup_{P\in\p} \,P(U_n \prec \tau^\infty) = 1$ 
then the coverage 
probability $\inf_{P\in\p}\,P(\tau_P \preceq U_n)\;=\;0$. This is also 
satisfied even in the smaller class $\p_f\subset\p$ of stationary ergodic measures having 
finite context tree.
\end{theorem}

The functional $T$ does however admit non-trivial lower confidence bounds on $\p$. In what follows, we construct a sequence of statistics which will be proved to be a non-trivial lower confidence bound and  a  consistent estimator of $T$ on $\p$, 
when $\sum_{v\in A^\star} \phi(v)<+\infty$. 

We will first define  a discrepancy measure between a sample $X_1,\dotsc, X_n$ and a 
measure $Q\in\p$. To do so, we need to introduce some more notation and definitions.
Given a string $w$, denote by $N_n(w)$ the number of occurrences of $w$ in the sample $X_1,\dotsc, X_n$; that is
\[
N_n(w) = \begin{cases} \sum_{i=0}^{n-|w|} \mathbf{1}\{X_{i+1}^{i+|w|} = w\} & n\geq |w|\\
0 & n< |w|.
\end{cases}
\]
If $N_{n-1}(w)>0$, we define for any $a\in A$ the estimated transition probability
\[
\hat p_n(a|w) := \frac{N_n(wa)}{N_{n-1}(w)} \,.
\]
Denote also by $C_{n-1}(w)$ the set of \emph{children}
of $w$ that appear in the sample at least once, that is
\[
C_{n-1}(w) = \{bw\colon b\in A\text{ and }N_{n-1}(bw)>0\}\,
\]
and by $S_n$ the set of all such strings $w$; that is 
\[
S_n = \{w\in A^\star\colon N_{n-1}(w) >0\} \,.
\]
Finally, for any context tree $\tau$, let
\[
\tau^*:=\{u\in A^\star\colon u\notin \bar\tau\}.
\]
Now, we can define our discrepancy measure as a function $d_n\colon A^n\times \p\to\R$ 
\[
d_n(X_1^n,Q) \;:= \;\max_{w\in S_n\cap\tau_Q^*}\,\{ \,N_{n-1}(w)\max_{a\in A}\,|\hat p_n(a|w) - Q(a|w)|\,\}
\] 
if $S_n\cap\tau_Q^*\neq\emptyset$. If $S_n\cap\tau_Q^*=\emptyset$ we define $d_n(X_1^n,Q)=0$. 

We are now ready to introduce the lower bound for the functional $T$.  
Given a constant $c>0$, for any $n\in\N$ let $T_n^c\colon A^n\to \T$ be defined
by  
\begin{equation}\label{eq:estimator}
T_n^c(X_1^n) \;=\; \inf\,\{\, \tau_Q\colon d_n(X_1^n,Q) \leq c\log(n) \,\}\,,
\end{equation}
where the infimum is taken with respect to the order $\prec$ between trees, and the logarithm is taken in base 2. 
Note that since the tree $\tau^\text{root}$ is the smallest element of $\T$ with respect to 
 $\prec$,  this infimum always exists. In Section~\ref{computation}
 we show how to practically compute  $T_n^c(X_1^n)$.
 
We now state the main result of this paper.

\begin{theorem}\label{lower_bound_T}
Given $0<\alpha<1$ and $n> 2$, for any $c$ satisfying
\begin{equation}\label{calpha}
(|A|-1)\Bigl(\frac{\log((|A|-1)/\alpha)}{\log(n)}+2\Bigr)  \;\leq \; c \;\leq\; \frac{n-1}{2|A| \log(n)}
\end{equation}
we have that the  statistic $T_n^c$ is a  non-trivial lower confidence 
bound for $T$ on $\p$, with nonparametric coverage probability of at least $1-\alpha$.
Moreover, if $\sum_{v\in A^\star} \phi(v)<\infty$, for any $c>2(|A|-1)$ the sequence  $\{T_n^c\}_{n\in\N}$ is a consistent estimator of $T$ on $\p$ and if $c>3(|A|-1)$ then $\{T_n^c\}_{n\in\N}$ is strongly consistent. 
\end{theorem}

\section{Computation and application of the lower confidence bound}\label{computation}

In this section we show how to compute the confidence bound \eqref{eq:estimator}
and we present a practical application of Theorem~\ref{lower_bound_T} to linguistic data.

\subsection{Tree lower bound algorithm}

Let $X_1,\dotsc,X_n$  be a given sample and let $c>0$  be a
fixed constant. To compute the tree $T_n^c(X_1^n)$, we will identify its nodes, i.e. the set $\bar{T}_n^c(X_1^n)$.  By definition, we know that  $w\in\bar{T}_n^c(X_1^n)$ if and only if every process $Q$ satisfying $d_n(X_1^n,Q)\leq c\log(n)$ has context tree with $w$ as a node. The following proposition gives a simple criteria to check whether or not we have to include a string $w$ in the set $\bar{T}_n^c(X_1^n)$. It relies on  two quantities, $l_n(w,a)$ and $u_n(w,a)$, which are defined for any $w\in S_n$ and any $a\in A$ by
 \begin{align}\label{eq:l}
 l_n(w,a) &= \max_{sw\in S_n} \;\Bigl\{ \hat p(a|sw) - \frac{c\log(n)}{N_{n-1}(sw)}\Bigr\}\\\label{eq:u}
 u_n(w,a) &= \min_{sw\in S_n} \;\Bigl\{ \hat p(a|sw) + \frac{c\log(n)}{N_{n-1}(sw)}\Bigr\}\,.
  \end{align} 
  \begin{proposition}\label{contprop}
Let $w$ be a finite string with $N_{n-1}(w)>0$. Then there exists a process $Q$
satisfying  $d_n(X_1^n,Q)\leq c\log(n)$ and having $w$ as a context if and only if
the following conditions hold
\begin{enumerate}
\item For any $a\in A$,  $l_n(w,a) \leq u_n(w,a)$.
\item $\sum_{a\in A} l_n(w,a) \,\leq\,1\,\leq\, \sum_{a\in A} u_n(w,a)$\,.
\end{enumerate}
\end{proposition}

We now give a simple algorithm (see Fig \ref{alg}) to construct the estimated tree. Let us explain how it works. Since every context tree has the root  $\lambda$ as node, then  $\lambda\in \bar{T}_n^c(X_1^n)$, and  we can inicialize the algorithm with $\lambda$. We then proceed iteratively  as follows, until we exhaust the set $S_n$. 

Suppose that a string $w$ has been included in $\bar{T}_n^c(X_1^n)$. If $C_{n-1}(w)\cap S_n\neq\emptyset$ and at least one  of the conditions of Proposition \ref{contprop} is not satisfied for $w$, this means that there do not exist processes $Q$ satisfying $d_n(X_1^n,Q)\leq c\log(n)$ and having $w$ as a context. In other words, all processes such that $d_n(X_1^n,Q)\leq c\log(n)$ has $w$
as a proper suffix of their contexts.  
Thus the set $C_{n-1}(w)\cap S_n$ must belong to $\bar{T}_n^c(X_1^n)$. On the other hand, if both conditions of Proposition \ref{contprop} are  satisfied for $w$, then there exists at least one process $Q$ such that $d_n(X_1^n,Q)\leq c\log(n)$ and having $w$ as a context. In this case we let $w$ be a context of $\bar{T}_n^c(X_1^n)$ and we stop checking its descendants (strings of the form $sw\in S_n$, with $s\in A^\star$). 

\begin{figure}[!htb]
\begin{center}
\framebox{
\begin{minipage}{10cm}
\vspace{2mm}
{\bf Tree lower bound (TLB) algorithm }
\begin{itemize}
\item[(1)] Initialise with $\bar T_n^c(X_1^n)\leftarrow\{\lambda\}$
 and $S\leftarrow \{\lambda\}$.
\item[(2)] While $S\neq \emptyset$, pick any $w\in S$  and do:
\begin{itemize}
\item[(a)] Remove $w$ from $S$;
\item[(b)] For any $a\in A$ compute the values $l_n(w,a)$ and $u_n(w,a)$ (see \eqref{eq:l} and \eqref{eq:u}).
\item[(c)] If for some $a\in A$, $u_n(w,a)<l_n(w,a)$ or if
\[
1\notin \Bigl[\sum_{a\in A} l_n(w,a)\,;\, \sum_{a\in A} u_n(w,a)  \Bigr]
\]
then add $C_n(w)$ to $\bar T_n^c(X_1^n)$ and to $S$.
\end{itemize}
\end{itemize}
\end{minipage}
}
\end{center}
\caption{Algorithmic steps to compute the lower bound in \eqref{eq:estimator}.}
\label{alg}
\end{figure}

\subsection{One-sided test of hypotheses for context trees}

In this subsection we present an application of the lower confidence bound introduced  in 
\eqref{eq:estimator}  to test a hypothesis about the context tree 
of codified  texts written in European Portuguese.  This  dataset, that is publicly available,  was first  analyzed in \citet{galves-et-al2012} where a method  to estimate a context tree was proposed and then applied to solve a linguistic conjecture about the rhythmic distinction between European and Brazilian Portuguese. The written texts were codified into the alphabet $A=\{0,1,2,3,4\}$ taking into account the stressed syllables and the boundaries of words; see \citet{galves-et-al2012} for details.  The European Portuguese context tree obtained in the cited work is the one shown in  Fig.~\ref{eptree}. Another analysis of the same dataset with similar results can be found in \cite{belloni-imbuzeiro2015}.

\begin{figure}[t!]
\subfigure[European Portuguese context tree.]{
\includegraphics[scale=0.8]{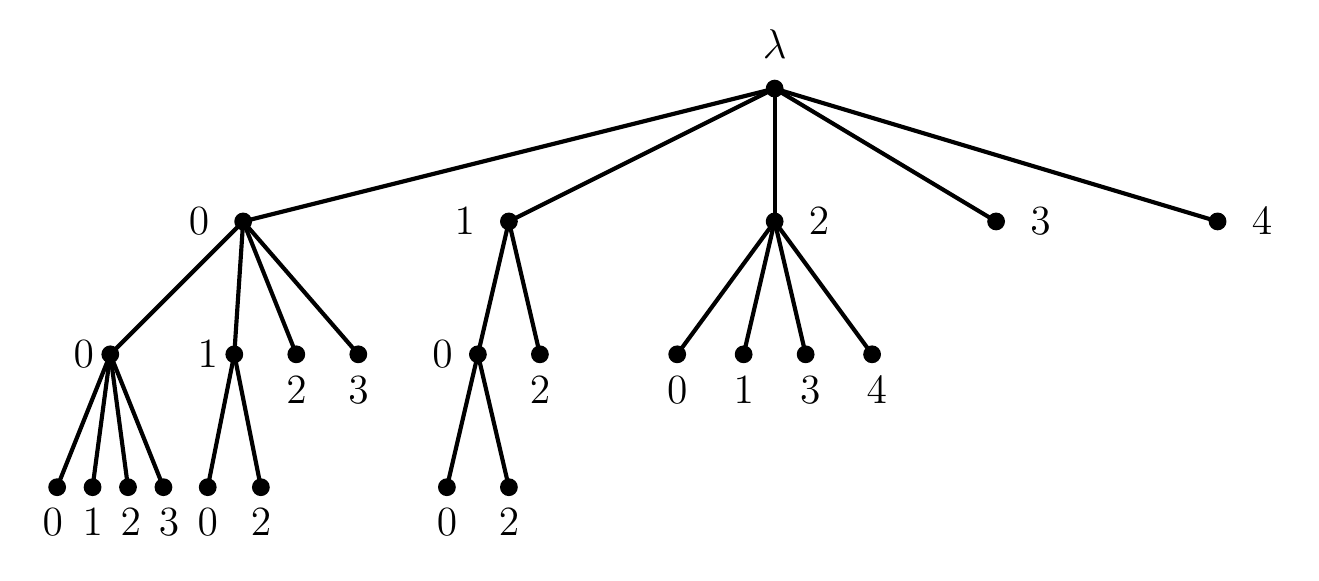}\label{eptree}}
\subfigure[Tree $\tau_0$ for the test of hypotheses.]{
\includegraphics[scale=0.8]{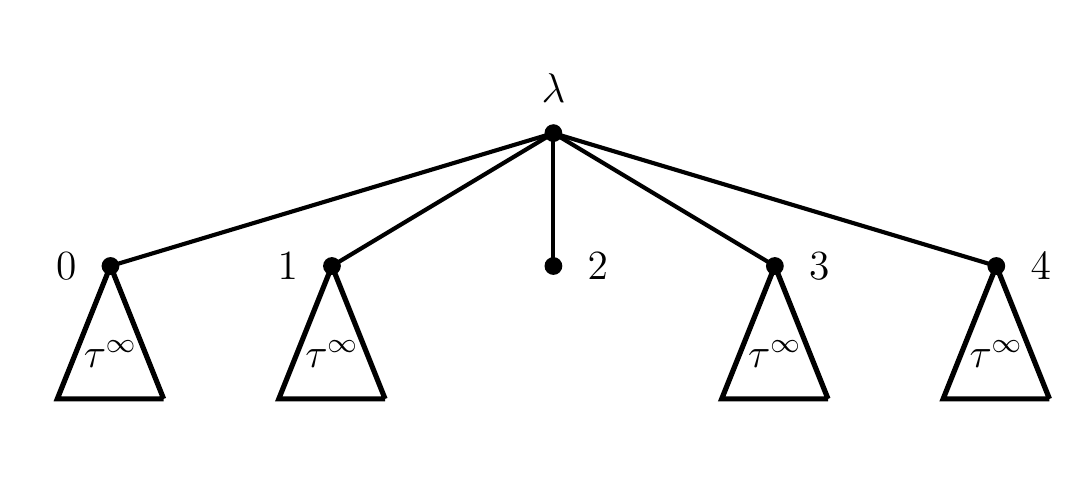}\label{eptreeinfty}}
\subfigure[Estimated lower bound.]{
\includegraphics[scale=0.8]{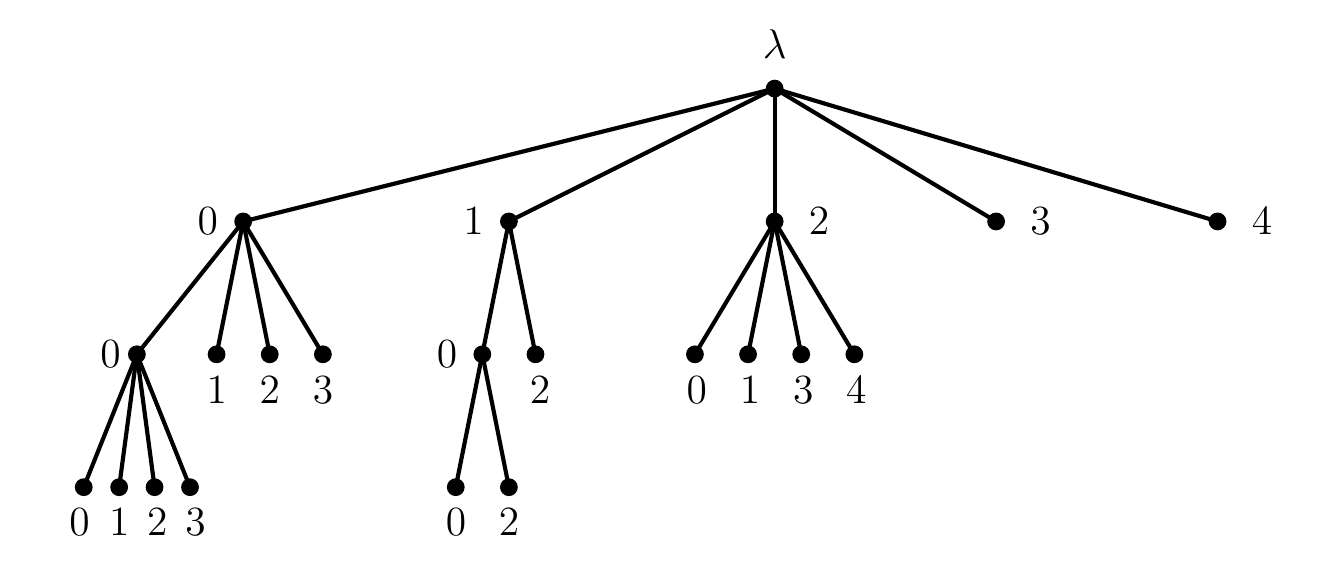}\label{tauhat}}
\caption{On top we show the European Portuguese context tree over the alphabet $A=\{0,1,2,3,4\}$ estimated from a corpus of codified written texts. 
On the middle we show a representation of the tree $\tau_0$ used for the definition of the test of hypotheses. The triangles with $\tau^\infty$ written inside represent infinite complete trees that ramify from  symbols $0,1,3$ and $4$. In other words, $\tau_0$ has a unique finite context, which is $2$. The bottom tree corresponds to the lower bound $T_n^c(X_1^n)$ computed on the same sample as the tree on top, with $\alpha=0.05$.}
\end{figure}

An interesting difference with a corresponding  linguistic interpretation between the two languages observed from the codified data was the ramification of string  ``2"  into the set of contexts ``02", ``12",  ``32''and ``42" that appears in the European Portuguese context tree in Fig.~\ref{eptree} (in the Brazilian Portuguese context tree this ramification did not occur and the string ``2" was identified as a context). A natural idea is then to test if there is enough evidence in the data supporting that the European Portuguese context tree ramifies from the sequence ``2" or not. 
  
It is well known that tests of hypotheses  can be constructed using confidence bounds. 
Let $\tau_0$ be a tree  and suppose we want to test the hypotheses 
\[
H_0: \tau_P \preceq \tau_0 \quad  vs.  \quad H_1: \tau_P\npreceq\tau_0\,.
\]
Given $n$ and $\alpha$, consider the test that rejects $H_0$ if and only if  $T_n^c(X_1^n)\in R=\{\tau\in \T: \tau\npreceq\tau_0\}$, with 
$c= (|A|-1)(\log((|A|-1)/\alpha)/\log(n)+2)$.
By Theorem \ref{lower_bound_T} we have that 
\[
\sup_{P:\tau_P\preceq \tau_0}P(T_n^c(X_{1}^{n})\npreceq \tau_0)\;\leq\; \sup_{P\in\p}P(T_n^c(X_{1}^{n})\npreceq \tau_P)\;\leq\; \alpha\,.
\]
Thus, the test defined by the rejection region $R$ has significance level $\alpha$
 for the hypotheses  $H_0: \tau_P\preceq \tau_0$ vs. $H_1: \tau_P\npreceq \tau_0$.

In our application the null hypothesis is defined by a tree $\tau_0$ having the string $``2"$ as a context. Since we impose no further condition, we let $2$ be the unique finite context of $\tau_0$,  that is
\[
\tau_0 =\{2\}\cup \{wa\colon w\in \tau^\infty, a\in A, a\neq 2\}\,.
\]
This tree is represented in Fig.~\ref{eptreeinfty}. We set the significance level $\alpha=0.05$,  and as our sample size is $n=107.761$ we have $c=9,513$.
The estimated tree with the TLB algorithm of Fig~\ref{alg} 
is given in Fig.~\ref{tauhat}. We see that $T^c_n(X_1^n)$ belongs to the rejection region $R$ therefore we reject the null hypothesis  at the significance level $\alpha=0.05$, confirming in this way the results of \cite{galves-et-al2012} about the ramification of sequence ``2" in the European Portuguese context tree.

The algorithm described in Fig.~\ref{alg} was coded in the {\ttfamily R} language and is available upon request.

\section{Proofs}\label{sec:proofs}

\begin{proof}[Proof of Lemma \ref{lema_baire}]
With respect to the weak topology, the set of all stationary probability measures over $(\Omega,\F)$ is a compact Hausdorff space \citep{shields1996}
and the subspace $\p$ of all stationary and  ergodic probability measures over $(\Omega,\F)$ is a $G_\delta$ set 
 \citep[Theorem 2.1]{category1961}. Therefore, $\p$ is a Baire space with the induced topology. 
 \end{proof}

\begin{proof}[Proof of Proposition \ref{non_consistency}]
The proof uses the same arguments of Lemma~1.1 in \citet{fraiman1999}. 
The difference is that  here we do not have independent random variables 
and  the space $\p$ is not a complete metric space with respect to $D$. But the same result
can be obtained in our setting, as we show in the sequel.
Recall that in the conditions of the proposition, there exists $R\in\R$ such that 
$|F(P)| \leq R$ for all $P\in\p$, and assume that  $\{F_n\}_{n\in\N}$ is a consistent estimator for $F$. Define 
\[
S_n = F_n I_{\{|F_n| \leq R\}} + sg(F_n) R I_{\{|F_n|>R\}},
\]
where $I$ is the indicator function and $sg$ is the sign of $F_n$.
It is not hard to show that
$\{S_n\}_{n\in\N}$ is also a consistent estimator for $F$, for details see  \citep[Lemma~1.1]{fraiman1999}.  
As for any $n\in\N$ the function $S_n$ is  bounded by $R$ we have that the convergence in probability to $F(P)$
implies convergence in mean. Therefore  we  have that 
\[
\phi_n(P)  := \E_P(S_n)  \;\to\; F(P)
\]
as $n\to\infty$.  Moreover,
\begin{align*}
|\phi_n(P) - \phi_n(Q)| &\;=\; \bigl|\,\sum_{x_1^n\in A^n} S_n(x_1^n)\,P(x_1^n) - \sum_{x_1^n\in A^n} S_n(x_1^n)\,Q(x_1^n)\,\bigr|\\
&\;\leq\; \sum_{x_1^n\in A^n}\, \bigl|S_n(x_1^n)\bigr|\, \bigl|\, P(x_1^n)-Q(x_1^n)\bigr|\\
&\leq\; R \,2^n D(P,Q)\,.
\end{align*}
Therefore, for each $n$, $\phi_n$ is uniformly continuous with respect to the weak topology (induced by $D$) on $\p$.  
Then, by Lemma~\ref{lema_baire} and the Baire's Cathegory Theorem, the function $F$ must be continuous on a dense subset of $\p$.
\end{proof}
 
To continue we need two basic lemmas that constitute  the core of all our negative results. 

\begin{lemma}\label{discontinuous1}
Any measure $P\in\p$ can be approximated with respect to $D$ by a sequence 
of measures $\{P_n\}_{n\in\N}$ in $\p$ each of which have as context tree a given tree $\tau$, with $\tau_P\preceq\tau$. In particular, $\tau$ can be  infinite.
\end{lemma}

\begin{proof}
We proceed in two steps, first we define a sequence of Markov measures $\{P^{[k]}\}_{k\in\N}$ converging  to $P$ and then for any $k\in\N$, we construct a sequence  of stationary ergodic  measures $\{P^{[k]}_i\}_{i\in\N}$ each of which have context tree $\tau$ and that converges to $P^{[k]}$. The conclusion of the proof then follows by a diagonal argument, since  convergence in $D$ (or in the weak topology) corresponds to convergence of the measure of cylinders \citep[Section I.9]{shields1996}. 

For any $k\in\N$, let   $P^{[k]}$ be the $k$-steps canonical Markov approximation of $P$, which is a Markov chain of order $k$ with transition probabilities  
\begin{equation}\label{pk}
P^{[k]}(a|a_{-k}^{-1}):=P(a|a_{-k}^{-1})\,,\,\,\,\,a \in A\,,\; a_{-k}^{-1}\in A^k\,.
\end{equation}
An important observation is that $\tau_{P^{[k]}}\preceq\tau_P$, since for any 
semi-infinite sequence $a_{-\infty}^{-1}\in A^\infty$ the length of the context of $P^{[k]}$ along $a_{-\infty}^{-1}$ is at most the length of the context of $P$. 
Moreover, it is well known that the sequence $\{P^{[k]}\}_{k\in\N}$ converges weakly to $P$ (see  \cite{rudolph1977} for instance), then the first step is proven. 

To continue, let us introduce the continuity rate of a process $\tilde P$ along a given past 
$a_{-\infty}^{-1}$, which is the non-increasing sequence
 $\{\beta^{\tilde P}_l(a_{-\infty}^{-1})\}_{l\in\N}$ defined as
\[
\beta^{\tilde P}_l(a_{-\infty}^{-1}):=\sup_{a,b_{-\infty}^{-1},c_{-\infty}^{-1}}|\tilde{P}(a|b_{-\infty}^{-1}a_{-l}^{-1})-\tilde{P}(a|c_{-\infty}^{-1}a_{-l}^{-1})|\,,\quad l\ge1.
\]
Observe that $\beta^{\tilde P}_{l-1}(a_{-\infty}^{-1})>0$ means $a_{-l}^{-1}\in\bar\tau_{\tilde P}$ and therefore  $\beta_l^{\tilde P}(a_{-l}^{-1})>0$ for all $l$ means that the infinite sequence $a_{-\infty}^{-1}\in\tau_{\tilde P}$. 
Let $\tilde{P}$ be a measure in $\p$  satisfying the following three conditions:
\begin{itemize}
\item[(i)] $\inf_{a_{-\infty}^{0}}\{\tilde P(a_0|a_{-\infty}^{-1})\} >0$. 
\item[(ii)] For any $a_{-\infty}^{-1}\in A^\infty$ and any $l\in\N$, $\beta_l^{\tilde{P}}(a_{-\infty}^{-1})>0$.  
\item[(iii)] $\sum_{l\in\N}\sup_{a_{-\infty}^{-1}}\beta^{\tilde{P}}_l(a_{-\infty}^{-1})<\infty$. 
\end{itemize}

It should be clear to the reader that such a measure $\tilde P\in\p$ can always be selected. An example of this is the observable chain in a Hidden Markov Model, that under simple assumptions satisfy conditions (i)-(iii) above, see for instance \cite{collet-leonardi-2014}.

Now consider any context tree $\tau$ such that  $\tau_P\preceq\tau$. For all $i\in\N$ define the kernel
\[
P^{[k]}_i(a|a_{-\infty}^{-1})=(1-1/i)\, P^{[k]}(a|a_{-k}^{-1})+ 1/i\, \tilde{P}(a|a_{-l}^{-1})\,,\; a\in A\,,\; a_{-l}^{-1}\in \tau\,.
\]
We have  $\inf_{a_{-\infty}^{0}}\{P^{[k]}_i(a_0|a_{-\infty}^{-1})\}\ge 1/i\, \inf_{a_{-\infty}^{0}}
\{\tilde{P}(a_0|a_{-\infty}^{-1})\}>0$ for any $i\in\N$. 
Thus, this kernel satisfies (i) and let us show that it also satisfies property (iii). For any  $a_{-\infty}^{-1}$ with $a_{-l}^{-1}\in \tau$ 
we have
\[
|P_i^{[k]}(a|b_{-\infty}^{-1}a_{-r}^{-1})-P_i^{[k]}(a|c_{-\infty}^{-1}a_{-r}^{-1})| = 1/i\, |\tilde{P}(a|b_{-\infty}^{-l-1}a_{-r}^{-1})-\tilde{P}(a|c_{-\infty}^{-l-1}a_{-r}^{-1})|
\]
for all $r<l$ or $|P_i^{[k]}(a|b_{-\infty}^{-1}a_{-r}^{-1})-P_i^{[k]}(a|c_{-\infty}^{-1}a_{-r}^{-1})| =0$ if $r\geq l$, for all $a\in A$ and  all $b_{-\infty}^{-1},c_{-\infty}^{-1}\in A^\infty$.
Conditions (i) and (iii)  ensure that there exists  a unique stationary ergodic measure having kernel $P_i^{[k]}$; see for instance \citet{bressaud/fernandez/galves/1999a,fernandez2002}. 

By the above observations, the contexts of $P_i^{[k]}$ are exactly the sequences in $\tau$, since $\tau_{P^{[k]}}\preceq\tau_P\preceq\tau$.
 Now, since $\{P_i^{[k]}\}_{i\in\N}$ converges uniformly to $P^{[k]}$ as $i\to\infty$ we also have  $\{P_k^{[k]}\}_{k\in\N}$ converging in $D$ to $P$ as $k$ diverges. 
 \end{proof}

\begin{lemma}\label{discontinuous2}
Any measure $P\in\p$ can be approximated with respect to $D$ by a sequence 
of measures $\{P_n\}_{n\in\N}$ in $\p$ each of which have a  finite context tree.
\end{lemma}

\begin{proof}
We prove this lemma using a similar two-steps argument as in the previous one. First, we use the same sequence of canonical Markov approximations $\{P^{[k]}\}_{k\in\N}$ defined in 
\eqref{pk} to approximate $P$. Second, as we do not know whether these Markov measures are ergodic or not, we construct, for any $k\ge1$, a sequence $\{P_i^{[k]}\}_{i\in\N}$ of ergodic Markov measures converging to $P^{[k]}$ when $i\to\infty$. The conclusion of the proof also follows from a diagonal argument. 

The construction of $P_i^{[k]}$ is also carried as in the previous lemma, by specifying the kernel of transition probabilities
\[
P^{[k]}_i(a|a_{-\infty}^{-1})= (1-1/i) \, P^{[k]}(a|a_{-k}^{-1})+ \frac{1}{i|A|}\,,\; a\in A\,,\; a_{-\infty}^{-1}\in A^\infty\,.
\]
Is is easy to see that this definition leads to a Markovian  (i.e.  with finite context tree) ergodic measure, and that  the sequence $\{P_i^{[k]}\}_{i\in\N}$  converges to $P^{[k]}$ when $i\to\infty$. As before we
have that $\{P_k^{[k]}\}_{k\in\N}$ converges  to $P$  when $k\to\infty$ and this concludes the proof.
\end{proof}

We are now ready to prove Theorem \ref{discontinuity}.

\begin{proof}[Proof of Theorem \ref{discontinuity}]
Assume that $\sum_{v\in \bar\tau^\infty} \phi(v) =+\infty$. Then Lemmas \ref{discontinuous1} 
and \ref{discontinuous2}  imply that any $P\in\p$ having $L(P)=0$ (respectively $L(P)=1$) is limit in $D$ of a sequence of measures $\{P_n\}_{n\in\N}$  in $\p$ satisfying $L(P_n)=1$ 
 (respectively $L(P_n)=0$) for all $n\in\N$.  In other words, the functional $L$ is discontinuous (with respect to the $D$-distance) at any point of $\p$. Together with  Proposition~\ref{non_consistency}, this   proves that $L$ is not consistently estimable on $\p$ when $\sum_{v\in \bar\tau^\infty} \phi(v) =+\infty$.
\end{proof}

\begin{proof}[Proof of Theorem \ref{teo-main}]
As we already mentioned, the proof of the \emph{if} part of the theorem follows from Theorem \ref{lower_bound_T} which states that $\{T_n^c\}_{n\in\N}$ is actually a consistent estimator of $\tau_P$. It remains  to prove the \emph{only if} part.
Assume $\sum_{v\in \bar\tau^\infty} \phi(v) = +\infty$   and suppose there exists $\{T_n\}_{n\in\N}$,  a consistent estimator of $T$ on $\p$. Define $L_n\colon A^n\to \{0,1\}$ by $L_n(x_1^n) = \mathbf{1}\{\sum_{v\in \bar{T}_n(x_1^n)}\phi(v) < +\infty\}$. We will prove that  $\{L_n\}_{n\in\N}$ is a consistent  estimator of $L$, which is a contradiction with Theorem \ref{discontinuity}, concluding the proof of the theorem.

Fix $P\in\p$.
As $\{T_n\}_{n\in\N}$ is consistent we have that for any $\epsilon>0$
\[
\lim_{n\to\infty} P(d_\phi(T_n(X_1^n),\tau_P) \leq  \epsilon) \;=\; 1\,.
\]
We will prove that for any $\epsilon>0$ the ball of center $\tau_P$ and radius $\epsilon$ contains only trees where $L$ is constant and equal to $L(\tau_P)$.  
By the definition of $d_\phi$, see \eqref{hamming}, for $\tau'\in \T$, 
\[
d_\phi(\tau_P,\tau') \;=\; \sum_{v\in\bar\tau'} \phi(v) \,+\,\sum_{v\in\bar\tau_P} \phi(v)-2\sum_{v\in\bar\tau'\cap\bar\tau_P}\phi(v).
\]
Then if $d_\phi(\tau_P,\tau') < \epsilon$ we have  $L(\tau_P) =1$ if and only if $L(\tau')=1$. 
Therefore
\[
\lim_{n\to\infty} P(\,L_n(X_1^n) = L(P)\,) \;\geq\; \lim_{n\to\infty} P(\,d_\phi(T_n(X_1^n),\tau_P) \leq  \epsilon\,) \;= \;1
\]
which proves that $L$ is consistently estimable on $\p$. But by Theorem~\ref{discontinuity}, $L$ is not consistently estimable on $\p$, which is a contradiction.
\end{proof}

\begin{proof}[Proof of Theorem \ref{teo_upper_bound}]
Suppose $U_n$ satisfies  
\[
\sup_{P\in\p}P(U_n  \prec \tau^\infty)=\; 1\,.
\]
Given $\delta>0$ choose a measure $P_{\delta}\in\p_f$  such that 
\begin{equation}\label{eq1}
P_{\delta}(  U_n \prec \tau^\infty)\;\geq \; 1-\frac{\delta}{3}\,.  
\end{equation}
This can always be done because the set $\p_f$ is dense in $\p$ (see Lemma~\ref{discontinuous2}). 
Let $\{\tau_k\}_{k\in\N}$ be an increasing sequence of finite trees and denote by $V_k$  the event
\[
V_k = \{\tau_k \prec U_n\}\,.
\]
We have $V_k\supset V_{k+1}$ for all $k\ge1$ and $\cap_{k\ge1} V_k = \{U_n =\tau^\infty\}$. 
Therefore 
\[
\lim_{k\to\infty} P_\delta(V_k) \;=\; P_\delta(U_n=\tau^\infty) \;\leq\; \frac13\delta\,.
\]
Now let $k^*$  be such that
\[
P_\delta(V_{k^*}) \;<\; \frac23\delta\,,
\]
and denote by $\tau$ the finite tree given by  $\bar\tau=\bar\tau_{k^*}\cup\bar\tau_{P_\delta}$. 
We have
\[
 P_{\delta}( \tau\npreceq  U_n )\;\geq \; P_{\delta}( \tau_{k^*}\npreceq  U_n )\;\geq \; 1-\frac{2}{3}\delta\,.
\]
By Lemma~\ref{discontinuous1}, there exists a measure $Q_\delta\in\p_f$ with 
$\tau_{Q_\delta}=\tau$ such that 
\[
D(P_{\delta},Q_\delta) \;<\; 2^{-n}\frac{\delta}{3}\,.
\]
Moreover we have 
\begin{align*}
P_{\delta} ( \tau\npreceq &\;U_n) \;-\;Q_\delta( \tau \npreceq U_n )\\
& =\; \sum_{x_1^n\in A^n} 
\mathbf{1}_{\{ \tau\npreceq U_n(x_1^n)\}}\,P_{\delta}(x_1^n) - \sum_{x_1^n\in A^n} \mathbf{1}_{\{\tau\npreceq U_n(x_1^n)\}}\,Q_\delta(x_1^n) \\
&\leq  \;\sum_{x_1^n\in A^n} \mathbf{1}_{\{\tau \npreceq U_n(x_1^n)\}}\,|\,P_{\delta}(x_1^n)
- Q_\delta(x_1^n)\,|\\
&\leq\;  2^nD(P_{\delta},Q_\delta) \;<\;\frac{\delta}{3}\,.
\end{align*}
Therefore 
\begin{equation*}
Q_\delta ( \tau \npreceq  U_n)\;> \; P_{\delta}( \tau\npreceq U_n ) \,-\,\frac{\delta}{3} \; \geq\; 1-\delta\,.
\end{equation*}
As $\delta$ is arbitrary we have just proved that 
\[
\inf_{Q\in\p} Q(\tau_Q \preceq U_n) \;=\; \inf_{Q\in\p_f} Q(\tau_Q \preceq U_n) \;=\; 0\,.\qedhere
\]
\end{proof} 

In order to prove Theorem~\ref{lower_bound_T} we will need  the following lemma. 

\begin{lemma}\label{prop:borne}
Given  $P\in\p$,  let $X_1,\dotsc,X_n$ be a sample of size $n$ with law $P$. Then 
for any constant $c>0$ we have 
\[
P\bigl( \,d_n(X_1^n,P) \leq\, c\log(n)\,\bigr) \;\geq\;  1- \frac{|A|-1}{n^{c/(|A|-1)-2}}\,.
\]
\end{lemma}

\begin{proof}
First note that if  $S_n\cap\tau_P^*=\emptyset$ then  the assertion of the lemma is trivial 
because in this case $d_n(X_1^n,P)=0$ and then $P(d_n(X_1^n,P) \leq\, c\log n)  =1$. Now suppose $S_n\cap\tau_P^*\neq\emptyset$
and let $w\in S_n\cap\tau_P^*$. We recall  the reader that all the logarithms are taken in base 2. This is not in fact a real restriction, as if the base is $r> 1$ we can replace $c$ by $c/\log_2(r)$ and the result holds as well.
First note that we can write 
\[
 \max_{a\in A} |\hat p_n(a|w) - p(a|w)| \;=\; \sum_{\substack{a\in A\\
 p(a|w)>\hat p_n(a|w)}}p(a|w) - \hat p_n(a|w)\,.
\]
For a proof of this equivalence see for instance \cite[Proposition~4.2 and Remark~4.3]{levin-et-al-book}. Therefore we have that the event
\[
B_n(w) = \{N_{n-1}(w)\, \max_{a\in A} |\hat p_n(a|w) - p(a|w)|>c\log(n)\}
\]
is included in the event
\begin{equation}\label{inclu1}
E_n(w) =  \bigcup_{a\in A} E_n(w,a)
\end{equation}
where
\[
E_n(w,a) = \Bigl\{(N_{n-1}(w)p(a|w) - N_n(wa))\mathbf{1}\{p(a|w)>\hat p_n(a|w) \} >  \frac{c\log(n)}{|A|-1}\Bigr\}\,.
\]
and the union \eqref{inclu1} has at most $|A|-1$ non empty sets. 
Now define the random variables 
\begin{align*}
W_n(w,a) &\;=\;  2^{N_{n-1}(w)\log(1+p(a|w))-N_n(wa)}\,,\quad n\geq |w|\,.
\end{align*}
Then, by the inequality $p(a|w) \leq \log_2(1+p(a|w))$ valid in the interval $(0,1]$
we have that for $a\in A$ such that $p(a|w)>\hat p_n(a|w)$,  the event $E_n(w,a)$
is included in the event  
\begin{equation}\label{inclu2}
F_n(w,a) = \Bigl\{W_n(w,a) > n^{c/(|A|-1)}\Bigr\}\,,
\end{equation}
As in \cite[Proposition~A.1]{garivier2011}, we will show that when $w\in \tau^*_P$,
the sequence $\{W_n(w,a)\}_{n\in \N}$ is a martingale with respect to  
the  filtration $\{\sigma(X_1,\dotsc,X_{n-1})\}_{n\in\N}$.
In fact, note that by the definition of $N_{n}(\cdot)$ we have that 
\begin{align*}
\E(2^{(N_{n+1}(wa) \,- N_n(wa))}|X_1,\dotsc,X_n)
&\;=\; \E(2^{\mathbf{1}\{X_{n+1-|w|}^{n+1}= wa\}}|X_1,\dotsc,X_n)\\
&\;=\; 2^{ (\mathbf{1}\{X_{n+1-|w|}^{n}= w\})\log(1+p(a|w))}\\
&\;=\; 2^{ (N_{n}(w) \,- N_{n-1}(w))\log(1+p(a|w))}
\end{align*}
which implies that $\E(W_{n+1}(w,a)|X_1,\dotsc,X_n)=W_n(w,a)$.
Thus $\E(W_{n}(w,a))=\E(W_{|w|}(w,a))=1$
and therefore, by Markov's inequality and a union bound 
 we have that
\begin{equation*}
P(E_n(w)) \;\leq\; (|A|-1) \sup_{a\in A} P( F_n(w,a) )\;\leq\;  \frac{|A|-1}{n^{c/(|A|-1)}}\,.
\end{equation*}
One more union bound over  $S_n\cap\tau_P^*$ yields    
\begin{align*}
P\bigl( \,d_n(X_1^n,P) > \,c\log n\bigr)\;& =\;  P\bigl(\;\cup_{w\in S_n} \, B_n(w)\;\bigr)\\
&\leq \; n^2\,\sup_{w\in S_n}\, \{\,P\bigl(\, E_n(w)\,\bigr)\,\}\\
&\leq \; \frac{|A|-1}{n^{c/(|A|-1)-2}}\,.
\end{align*}
\end{proof}

\begin{proof}[Proof of Theorem \ref{lower_bound_T}]
Observe that  for any $P\in\p$  the event
 $\{ d_n(X_1^n,P) \leq c\log(n)\}$ implies $\{ T_n^c(X_1^n) \preceq \tau_P\}$. Therefore, by Lemma~\ref{prop:borne} and the condition on $c$  we have
\begin{align*}
P(\,T_n^c(X_1^n) \preceq \tau_P\,)& \;\geq\; P (d_n(X_1^n,P) \leq c\log(n))\\
& \;\geq\; 1-
(|A|-1)/n^{c/(|A|-1)-2}\\
&\;\geq \; 1-\alpha\,.
\end{align*}
To prove that $T_n^c(X_1^n)$ is not trivial we have to prove that 
\begin{equation}\label{eq:nontrivial}
\sup_{P\in\p} P( \tau^\text{root} \prec T_n^c(X_1^n))=1\,.
\end{equation}
For that consider the transition matrix $Q^{\epsilon}$ on $A=\{1,\ldots,k\}$  defined by
\begin{align*}
Q^{\epsilon}(i,i)& \;=\;1-Q^{\epsilon}(i,i+1)\;=\;\epsilon\,,\quad \text{ for $i=1,\ldots,k-1$, \, and }\\ Q^{\epsilon}(k,k)&\;=\;1-Q^{\epsilon}(k,1)\;=\;\epsilon\,.
\end{align*}
The parameter $\epsilon$ will be chosen adequately later. Observe that if $\epsilon>0$ there exists a unique stationary  ergodic Markov chain $P$ specified by $Q^{\epsilon}$. It is also easy to see, by symmetry, that $P(i)=1/k$ for any $i=1,\ldots, k$. Now for any $n\geq 3$ denote by $\mathcal{M}_n$ the set of strings $x_1^n\in A^n$ such that 
 $x_{i+1}= x_{i}+1$ for $i=1,\ldots,n-1$. 
  Observe that there are exactly $k$ such strings, independently of $n$, each beginning in  a different symbol of $A$. Moreover, each of these strings have equal measure
\[
P(x_{1}^{n})=\frac{1}{k}(1-\epsilon)^{n-1}.
\]
On the other hand, thanks to Proposition~\ref{contprop} and the  TLB algorithm in Fig.~\ref{alg}, the $k$ trees $\{T_n^c(x_1^n):x_{1}^{n}\in \mathcal{M}_n\}$ will be different from 
$\tau^\text{root}$, because if $c\leq (n-1)/ 2|A|\log(n)$, as $N_{n-1}(x_1) \geq N_{n-1}(x_2)\geq \lceil \frac{n-1}{|A|}\rceil$ we will have 
\[
u_n(x_2,x_2) \;\leq\; \hat p_n(x_2|x_2) + \frac{c\log(n)}{N_{n-1}(x_2)} <\; 1/2.
\]
and
\[
l_n(x_1,x_2) \;\geq\; \hat p_n(x_2|x_1) - \frac{c\log(n)}{N_{n-1}(x_1)} \;>\; 1/2\,.
\]
In other words,
\[
P( \tau^\text{root} \prec T_n^c(X_1^n)) \; \geq \; P( X_1^n\in\mathcal{M}_n)=k\times\frac{1}{k}(1-\epsilon)^{n-1}\,.
\]
To conclude, observe that for any $\delta>0$, we can take $\epsilon=1-(1-\delta)^{1/(n-1)}$, and we get
\[
P( \tau^\text{root} \prec T_n^c(X_1^n))\;\ge\;1-\delta
\]
proving that \eqref{eq:nontrivial} holds.

Now we will prove the consistency of the estimator 
$\{T_n^c(X_1^n)\}_{n\in\N}$ for any $c>2(|A|-1)$, by showing that for any $P\in\p$ and any $\epsilon>0$ the event $d_\phi(T_n^c(X_1^n),\tau_P)\leq\epsilon$ occurs with probability converging to 1 as $n\to\infty$. 
To begin, notice that by Lemma~\ref{prop:borne} we have that 
\[
P(d_n(X_1^n,P)\leq c\log n ) \;\geq\; 1- \frac{|A|-1}{n^\delta}\quad\text{  for all }n\geq 1\,, 
\]
where $\delta=c/(|A|-1)-2 > 0$. 
This implies, by the definition of $T_n^c(X_1^n)$,  that 
\begin{equation}\label{boundsum}
P(T_n^c(X_1^n) \preceq  \tau_P ) \;\geq\; 1- \frac{|A|-1}{n^\delta}\quad\text{ for all }n\geq 1\,.
\end{equation}
Now, recall that in the conditions of the theorem $\sum_{v\in \bar\tau^\infty} \phi(v)<\infty$, and take $k\in\N$ such that 
\begin{equation}\label{eps}
\sum_{u\in \bar\tau_P\colon |u|>k} \phi(u)  \;<\; \epsilon \,.
\end{equation}
Thus we have with probability at least $1-(|A|-1)/n^\delta$ that 
\begin{align}\label{dist}
 d_\phi(T_n^c(X_1^n),\tau_P) \;&\leq\;  \sum_{u\in \bar\tau_P|_k\setminus \bar T_n^c(X_1^n)} \phi(u)  + \sum_{u\in \bar\tau_P\colon |u|>k} \phi(u)\notag\\
 &\leq\;  \sum_{u\in \bar\tau_P|_k\setminus \bar T_n^c(X_1^n)} \phi(u) \; +\; \epsilon\,.
 \end{align}
Therefore it is enough to prove that  $\bar\tau_P|_k\setminus \bar T_n^c(X_1^n)=\emptyset$  with probability converging to 1 as $n\to\infty$, or what is stronger, with probability equal to 1 for $n$ sufficiently large, a fact that we refer as to occur \emph{eventually almost surely} or \emph{e.a.s.} for short. 
To prove this last assertion, for any  $v\in \bar\tau_P|_k$ we will show that  the set $\{Q\colon d_n(X_1^n,Q) \leq c\log n\}$ is included in the set $\{Q\colon v\in\bar\tau_Q\}$ e.a.s. As the set  $\bar\tau_P|_k$ is finite we will have  $\{Q\colon d_n(X_1^n,Q) \leq c\log n\}\subset\{Q\colon\tau_Q\succeq \tau_P|_k\}$  e.a.s. and therefore $T_n^c(X_1^n)\succeq \tau_P|_k$  e.a.s., which in turns implies that
$\bar\tau_P|_k\setminus \bar T_n^c(X_1^n)=\emptyset$ e.a.s. 
 So, let $v\in \bar\tau_P|_k$,  $v=v_1^j$, and denote by $v'$ its largest 
proper suffix, that is $v'= v_2^j$.  It can be shown that we can always find a symbol 
$b\in A$ and another  finite string $w\in A^\star$
 such that  the following conditions hold
 \begin{itemize}
 \item[(i)] $P(a|v_1wv') \neq P(a|bwv')$ for some $a\in A$ and 
 \item[(ii)] $v$ is a suffix (proper or not) of $v_1wv'$.
 \end{itemize}
 If $v$ is a context for $P$, i.e if $v\in\tau_P\cap\bar\tau_P|_k$ then it is enough to take $w=
 \lambda$ and $b\neq v_1$ satisfying (i) (such $b$ must always exist because $v$ is a 
 context).  In this case we have $v=v_1wv'$ and (ii) is also satisfied. On the other hand, if $v$ 
 is not a context then it is an internal node of $\tau_P$ and we must have some $w=w'v_1$, with $w'\in A^\star$, and $b\in A$ such that (i)-(ii) are satisfied. If not, this would imply that for any $u\in A^\star$  
 $P(\cdot |uv)=P(\cdot|v)$, contradicting the fact that $v$ is a proper suffix of a context.
Using the triangle inequality we have, for any $Q\in\p$, that 
\begin{align}\label{eq:triangle}
|Q(a|v_1wv')-&Q(a|bwv')|\;\ge\; | P(a|v_1wv')- P(a|bwv')|\notag\\
&- |P(a|v_1wv')-\hat p_n(a|v_1wv')|-|\hat p_n(a|v_1wv')-Q(a|v_1wv')|\notag \\
&-|Q(a|bwv')-\hat p_n(a|bwv')|-|\hat p_n(a|bwv')-P(a|bwv')|\,.
\end{align}
Now, by ergodicity  we have that $P$-almost surely, for any finite sequence $s\in A^\star$ 
\begin{equation}\label{eq:ergo}
\Bigl|\frac{N_n(s)}{n} - P(s)\Bigr|\to 0
\end{equation}
when $n\to \infty$,  which in turns implies that for any $s\in A^\star$
\begin{equation}\label{eq:ergo2}
|\hat p_n(a|s) - P(a|s)|\to0\,.
\end{equation}
In particular by \eqref{eq:ergo}, for any finite sequence $s\in A^\star$ we will have $N_n(s) \geq nP(s)/2$ e.a.s. This fact, together with \eqref{eq:ergo2} and the inequality \eqref{eq:triangle} implies that for a sufficiently large $n$,  any measure  $Q$ such that $d_n(X_1^n,Q) \leq c\log n$ 
 will satisfy 
\begin{equation}\label{eq:cont}
|Q(a|v_1wv')-Q(a|bwv')|\; >\; 0
\end{equation}
and thus $v\in\bar\tau_Q$. 
Therefore $\{Q\colon d_n(X_1^n,Q) \leq c\log n\}\subset\{Q\colon v\in\bar\tau_Q\}$ and 
$\bar\tau_P|_k\setminus \bar T_n^c(X_1^n)=\emptyset$ e.a.s, as required, showing that
\[
P(d_\phi(T_n^c(X_1^n),\tau_P) \leq \epsilon)\;\to \; 1\quad\text{ when }n\to\infty\,.
\]
To finish the proof we only emphasize that if $c>3(|A|-1)$ then $\delta>1$ and  therefore the bound in \eqref{boundsum} is summable. This fact together with the Borel-Cantelli lemma implies that 
$\T_n^c(X_1^n) \preceq \tau$ e.a.s. Therefore, as the other inclusion 
$\tau|_k\preceq \T_n^c(X_1^n)$ also holds e.a.s, by $\eqref{dist}$ we will
have that for all $n$ sufficiently large 
\[
d_\phi(T_n^c(X_1^n),\tau_P) \;\leq\; \epsilon
\]
with probability one.
\end{proof}

\begin{proof}[Proof of Proposition~\ref{contprop}]
First suppose there exists a process $Q$ satisfying  \linebreak$d_n(X_1^n,Q)\leq c\log(n)$ and having $w$ as a context. Then for any $s\in A^*$ such that $sw\in S_n\cap\tau_Q^*$ and any $a\in A$ we have that 
\[
|Q(a|sw) - \hat p(a|sw)|  \, \leq\, \frac{c\log(n)}{N_{n-1}(sw)}\,.
\]
Therefore for any $a\in A$ and any $s\in A^*$ such that $sw\in S_n\cap\tau_Q^*$  we have
\begin{equation}\label{qwa}
 \hat p(a|sw) - \frac{c\log(n)}{N_{n-1}(sw)}  \, \leq\, Q(a|w)  \, \leq\, \hat p(a|sw)+  \frac{c\log(n)}{N_{n-1}(sw)}
\end{equation}
because $w$ is a context for $Q$ which implies 
\[
l_n(w,a) \,\leq\, Q(a|w)  \, \leq\, u_n(w,a)
\]
by maximizing (minimizing) with respect to $s$ the left (right) side of \eqref{qwa}, and the first condition 
is proven. Now, by summing this last inequality over $a\in A$ we obtain that 
$\sum_{a\in A}l_n(w,a)$ and $\sum_{a\in A}u_n(w,a)$ must satisfy the second condition as well.  

Now suppose 1. and 2. hold. Then it is easy to see that these conditions allow us  to choose some positive values $Q(a|w)$ in the interval $(l_n(w,a);u_n(w,a))$ for any $a\in A$ in such a way that $\sum_{a\in A} Q(a|w)=1$. Then, it is straightforward to construct a stationary ergodic Variable Length Markov chain having $\{Q(a|w)\}_{a\in A}$ as the conditional distribution of the next symbol in the sequence given the past string $w$, and such that $w$ is a context for $Q$. 
The other contexts  $v\in \tau_Q\setminus\{w\}$ can be 
chosen  arbitrarily large in such a way that $N_{n-1}(v)=0$, implying that $S_n\cap\tau^*_Q=\{sw\colon N_{n-1}(sw)>0\}$. The conditions on  $\{Q(a|w)\}_{a\in A}$ implies that for all $sw\in S_n$ and all $a\in A$ we have
\[
\hat p_n(a|sw) - \frac{c\log(n)}{N_{n-1}(sw)} \;\leq\; Q(a|sw) \;\leq\; \hat p_n(a|sw) + \frac{c\log(n)}{N_{n-1}(sw)} 
\]
and therefore 
\[
d_n(X_1^n,Q)\; =\; \max_{v\in S_n\cap\tau^*_Q } \,\{N_{n-1}(v) \max_{a\in A}|\hat p_n(a|v) - Q(a|v)|_1\}\; \leq\; c\log(n)\,. \qedhere
\]
\end{proof}

\section*{Discussion}

The main contribution of this work is  the introduction of a lower confidence bound for the context tree in the class of stationary ergodic probability measures over $A^\Z$, with $A$
a finite alphabet. We derive an explicit formula for the coverage probability of this confidence bound, based on a  martingale deviation inequality developed in  \citet{garivier2011}, and we show the almost sure convergence of this estimator  with respect to the Hamming distance $d_\phi$, when $\phi$ is summable.  To our knowledge, this is the first lower confidence bound for context trees and it is also the first strong consistent estimator that do not restrict the length of the estimated contexts, as for example  does the BIC context tree estimator in \citet{csiszar2006} that only allows candidate contexts of length $o(\log n)$. Using only topological arguments  we  also prove that if $\phi$ is not summable then  there exists no  consistent estimator of the context tree in the class of stationary and ergodic processes. This is not the case in the class of processes having 
 finite context trees because  in this case  the BIC estimator is strongly consistent. 
On the other hand we also prove  that it is not possible to obtain nonparametric upper  confidence bounds even in the smaller class of processes having finite context trees, because any process can be approximated in $D$  by stationary ergodic processes having arbitrary large context trees, as shown in Lemma~\ref{discontinuous1}. 
We also show in this work a practical application of the lower confidence bound to  test a hypothesis involving the presence of finite contexts in codified written texts of European Portuguese. We support at the confidence level of 95\% the results obtained in \cite{galves-et-al2012} for this dataset, where only point estimation of the context tree was addressed. 

\section*{Acknowledgements}
We are thankful to Antonio Galves, Ricardo Fraiman and Miguel Abadi
for interesting discussion on the subject and to the Associate Editor and an anonymous referee for suggestions to improve the exposition of the results. 

\bibliographystyle{imsart-nameyear}
\bibliography{references}  

\end{document}